\documentclass[aip, cha, twocolumn, groupedaddress, showpacs, floatfix, letterpaper, 10pt]{revtex4-2}

\usepackage{amsmath, amssymb, amsthm}
\usepackage[english]{babel}
\usepackage[latin1]{inputenc}
\usepackage{times}
\usepackage{todonotes,color}
\usepackage{xy}
\usepackage{graphicx}
\usepackage{verbatim}
\usepackage{enumerate}
\usepackage{url}

\newcommand{\eps}{\varepsilon}

\newcommand{\cG}{\mathcal{G}}
\newcommand{\cM}{\mathcal{M}}

\newcommand{\DZ}{\mathrm{DZ}}
\newcommand{\Tg}{\mathrm{T}}
\newcommand{\ud}{\mathrm{d}}

\newcommand{\res}{\mathrm{res}}
\newcommand{\iin}{\mathrm{in}}

\newcommand{\Er}{\bbR}
\newcommand{\bbT}{\mathbb{T}}
\newcommand{\bbZ}{\mathbb{Z}}
\newcommand{\bbR}{\mathbb{R}}

\newcommand{\vp}[2]{\left\langle#1, #2\right\rangle}

\newcommand{\nud}{\nu}
\newcommand{\norm}[1]{\left\|#1\right\|}
\newcommand{\abs}[1]{\left|#1\right|}

\newcommand{\rset}[2]{\left\lbrace\, #1\,\left|\;#2\right.\right\rbrace}

\newcommand{\set}[2]{\rset{#1}{#2}}
\newcommand{\sset}[1]{\left\lbrace #1\right\rbrace}

\newtheorem{prop}{Proposition}
\newtheorem{lemm}{Lemma}

\newtheorem{defn}{Definition}

\theoremstyle{remark}
\newtheorem{rem}{Remark}

\newcommand{\hf}{\mathrm{h}}
\newcommand{\gf}{{\mathrm{g}}}
\newcommand{\ff}{\mathrm{f}}
\newcommand{\If}{\mathrm{I}}
\newcommand{\Ns}{\mathrm{N}}
\newcommand{\PR}{\mathrm{PR}}

\DeclareMathOperator{\grad}{grad}

\begin{document}


\title{Dead zones and phase reduction of coupled oscillators} 



\author{Peter Ashwin${}^{1,2}$, Christian Bick${}^{1,3}$, Camille Poignard${}^{1,2}$}
\affiliation{
${}^1$Centre for Systems, Dynamics and Control and Department of Mathematics, University of Exeter, Exeter EX4 4QF, UK\\
${}^2$EPSRC Centre for Predictive Modelling in Healthcare, University of Exeter, Exeter EX4 4QF, UK\\
${}^3$Department of Mathematics, Vrije Universiteit Amsterdam, De Boelelaan 1111, Amsterdam, the Netherlands
}

\date{\today}

\begin{abstract}
A dead zone in the interaction between two dynamical systems is a region of their joint phase space where one system is insensitive to the changes in the other. These can arise in a number of contexts, and their presence in phase interaction functions has interesting dynamical consequences for the emergent dynamics. In this paper, we consider dead zones in the interaction of
general coupled dynamical systems. For weakly coupled limit cycle oscillators, we investigate criteria that give rise to dead zones in the phase interaction functions. We give applications to coupled multiscale oscillators where coupling on only one branch of a relaxation oscillation can lead to the appearance of dead zones in a phase description of their interaction.
\end{abstract}

\maketitle

\allowdisplaybreaks

\begin{quotation}
The collective dynamics of networks of coupled units depends not only on which units are connected but also on how they are connected. In certain physical systems one can observe that network connections may be state dependent in the sense that links can be temporarily disabled. For example, in networks of neural oscillators, a unit may have a refractory period and be insensitive to input after sending an action potential. This can be mathematically captured by the concept of a ``dead zone'' in the coupling function. Building on recent work~\cite{AshBicPoi2019}, we generalize the notion of a dead zone to general network dynamical systems. We focus on the case of coupled oscillator networks. Even if the coupled nonlinear oscillatory processes do not possess dead zones, the effective phase dynamics for weak coupling may possess a dead zone. On the other hand, dead zones of interaction for limit cycle oscillators may or may not become dead zones for a phase reduced system. We make this explicit for networks of coupled relaxation oscillators where the oscillators are shaped by the separation of time scales and the geometry of critical manifolds.
\end{quotation}

\section{Introduction}

The collective dynamics of a network of~$N$ coupled dynamical units depends not only on the network structure (i.e., which unit is coupled to which other unit) but also on the functional form of the interactions~\cite{Stankovski2017}. It is well known that various types of dynamical effects such as chaos and synchronization can be understood using such models even for relatively small numbers of oscillators \cite{anishchenko2007nonlinear}.

However, many biological oscillators are insensitive to inputs in a particular state~\cite{UriTei2019} and this may lead to effects that are not typical for ``generic'' coupling. Neural oscillators with a refractory period behave similarly: After emitting an action potential, there is a ``refractory'' period in which the neuron does not react to further input~\cite{AshCooNic2016}.

Even if the network connections themselves remain fixed, the functional form of the network interactions can lead to \emph{effective} decoupling of nodes for certain states of the network dynamical system. In this case, the interaction function has \emph{dead zones}, which gives rise to an \emph{effective} interaction graph as a subgraph of the underlying structural network. In a recent paper~\cite{AshBicPoi2019}, we explored the dynamical consequences of dead zones for a class of network dynamical systems. Specifically, we formalized the notion of a dead zone and the effective interaction graph for averaged phase oscillator networks in terms of their phase interaction functions. 
Such phase oscillator networks can be derived from networks of nonlinear oscillators through a phase reduction to describe their dynamics.
Here we consider dead zones in more general networks of nonlinear oscillators. Moreover, we elucidate the question of how dead zones may emerge in the effective phase dynamics in a phase reduction. We note this may emerge in the phase description, whether or not the nonlinear oscillators have dead zones in their coupling.

For a general network dynamical system, we assume that the phase space of each node is a smooth manifold~$\cM$
with tangent bundle~$\Tg\cM$, and denote by~$\Tg_x\cM$ the tangent space at $x\in\cM$: Generally, this will be either~$\Er^d$ or the torus $\bbT := \Er/2\pi\bbZ$.
Consider an additively coupled network dynamical system with~$N$ similar nodes, where the state of node~$k$ is determined by $x_k\in\cM$ and by selective interactions. Specifically, the network dynamics on~$\cM^N$ is determined by the ordinary differential equation (ODE)
\begin{equation}\label{eq:NetworkA}
\dot x_k := \frac{\ud x_k}{\ud t} = \ff_k(x_k) + \eps\sum_{j=1,j\neq k}^N A_{jk} \gf_{jk}(x_j, x_k),
\end{equation}
where the functions~$\ff_k:\cM\to \Tg\cM$ determine the intrinsic node dynamics, $\gf_{jk}:\cM^2\to \Tg\cM$ for $j\neq k$ are the \emph{coupling functions}, $A_{jk}\in\{0,1\}$ are the coefficients of the adjacency matrix that determine the network structure, and the parameter $\eps>0$ is a coupling strength. For all-to-all coupling of identical units, the node dynamics as well as the coupling are assumed to be identical and $A_{jk}=1$ for $j\neq k$. In this special case, Equation~\eqref{eq:NetworkA} reads
\begin{equation}\label{eq:Network}
\dot x_k = \ff(x_k) + \eps\sum_{j=1,j\neq k}^N \gf(x_j, x_k),
\end{equation}
where $\ff:\cM\to \Tg\cM$ and $\gf:\cM^2\to \Tg\cM$. We will concentrate on systems of the form~\eqref{eq:Network} for the remainder of this paper. It is straightforward to generalize some of these result to networks~\eqref{eq:NetworkA} but the notation becomes more cumbersome~\cite{AshBicPoi2019}.

Dead zones for~\eqref{eq:Network} are characterized by a vanishing interaction function~$\gf$. An open subset of~$\cM^2$ is a dead zone for~\eqref{eq:Network} if~$\gf$ is identically zero on this subset. We make this notion precise below. On the one hand, the coupled phase oscillator networks considered in Reference~\onlinecite{AshBicPoi2019} of the form
\begin{equation}\label{eq:NetworkOscAvg}
\dot \theta_k = \omega + \eps\sum_{j=1}^N \hf(\theta_j-\theta_k)
\end{equation}
for $\theta_k\in\bbT$ are a special case of~\eqref{eq:Network} with $\cM=\bbT$ and $x_k=\theta_k$. In this context, the interaction between oscillators~$j$ and~$k$ is determined by the coupling function~$\hf$: A dead zone is an open connected set of phase differences where $\hf=0$. On the other hand, under suitable assumptions, system~\eqref{eq:NetworkOscAvg} can be \emph{derived} from a nonlinear oscillator network~\eqref{eq:Network} with $x_k\in\Er^d$ and serve as a description of the effective dynamics.
Then the coupling function~$\hf$ can be derived from~\eqref{eq:Network} in terms of the oscillator properties~$\ff$ and the interactions~$\gf$.

In this paper we focus on the latter case and tackle the relationship between dead zones in the nonlinear oscillator system~\eqref{eq:Network} with~$x_k\in\Er^d$ and dead zones in the phase oscillator network~\eqref{eq:NetworkOscAvg}. For example, does the existence of a dead zone for~\eqref{eq:Network} imply the existence of a dead zone for~\eqref{eq:NetworkOscAvg}? Are there ways that the effective phase dynamics~\eqref{eq:NetworkOscAvg} have a dead zone while~\eqref{eq:Network} does not?

For the remainder of this introduction we generalize some concepts about dead zones to the setting~\eqref{eq:Network}. We focus on the case of separable coupling functions where the coupling interaction can be written as a product of response and input functions. Section~\ref{sec:Oscillators} considers questions related to when dead zones for the interactions of weakly coupled oscillators result in dead zones for the averaged phase equations~\cite{AshBicPoi2019}. Section~\ref{sec:multiscale} examines weakly coupled multiscale oscillators and states some explicit conditions that result in dead zones for interaction. This continues with discussion of an example of coupled FitzHugh--Nagumo oscillators with coupling through the fast variable.  These mechanisms are also relevant in contexts beyond oscillators, for example, for the synchronization of chaotic systems where phase information can be extracted~\cite{anishchenko1992synchronization}.
We finish with a brief discussion in Section~\ref{sec:discuss}.

\subsection{Dead zones of interaction for coupled dynamical systems}

The notions of dead zones and effective coupling graphs considered in Reference~\onlinecite{AshBicPoi2019} generalize naturally to network dynamical systems~\eqref{eq:Network}. Suppose that~$A$ is a set and~$X$ a vector space. Given a function~$f:A\to X$ write 
\[\Ns(f) = \set{x\in A}{f(x)=0}\]
for the zero set of~$f$.

\begin{defn}\label{def:DZ}
A \emph{dead zone} of the coupling function~$\gf$ is a maximal connected open set $U\subset\cM^2$ such that $\gf(U)=0$.
\end{defn}

For a given coupling function~$\gf$ let~$\DZ(\gf)$ denote the union of all dead zones. The coupling function~$\gf$ has \emph{simple dead zones} if~$\DZ(\gf)$ is connected, that is, there is exactly one dead zone.
Effective coupling can now be encoded by a graph as in Reference~\onlinecite{AshBicPoi2019}; for completeness, we generalize the notion of an effective coupling graph to~\eqref{eq:Network}.

\begin{defn}
The \emph{effective coupling graph~$\cG_\gf(x)$ at $x\in\cM^N$} is a directed graph on~$N$ vertices with edges
\[
E(\cG_\gf(x)) = \set{j\to k}{(x_j,x_k) \not\in\DZ(\gf)}.
\]
\end{defn}

\subsection{Separable coupling functions}
\label{sec:sep}

For many systems of interest, the coupling function~$\gf$ has additional properties. If~$X$ is a vector space, we denote by
$v\odot w$ the Hadamard (element-wise) product of $v,w\in X$, i.e., the vector in~$X$ with components 
$[v\odot w]_j=v_jw_j$.
We say a coupling function~$\gf$ for~\eqref{eq:Network} is \emph{separable} if it can be written as 
\[
\gf(x_j,x_k) = \gf^\iin(x_j)\odot\gf^\res(x_k)
\]
where $\gf^\iin:\cM\to \Tg\cM$ the \emph{input function} and $\gf^\res:\cM\to \Tg\cM$ is the \emph{response function}. Many commonly studied network dynamical systems have separable coupling functions. These include:

\paragraph{Phase oscillator networks.} The state of a phase oscillator is given by $x_j\in \cM = \bbT$ for each~$j$. If $\ff_k = \omega\in\Er$ and separable $\gf(x_j,x_k) = Z(x_k)I(x_j)$ with $Z:\bbT\to\Er$ and $I:\bbT\to\Er$ then the dynamics are determined by
\[
\dot x_k = \omega + \sum_{j\neq k}Z(x_k)I(x_j),
\]
i.e., with $\gf^\res=Z$, $\gf^\iin=I$. Such networks arise from phase reductions; we will explore these further in Section~\ref{sec:Oscillators}.

\paragraph{State-independent coupling.}
The master stability function approach~\cite{Pecora1998} is a classical tool to determine the stability of synchrony in networks of the form
\begin{align*}
\dot x_k = \ff(x_k) + \eps\sum_{j=1}^N A_{kj}\gf(x_j),
\end{align*}
with all the~$x_k$ in $\Er^d$. The coupling function is separable with $\gf^\iin=\gf$ and $\gf^\res = 1$.

\paragraph{Diffusive coupling.}
For network dynamical systems with ``diffusive coupling'' the dynamics of a given node is depends on the difference between its state and the states of the nodes it receives input from.
Specifically, the state of node~$k$ is determined by $x_k\in\Er^d$ and evolves according to
\begin{align}\label{eq:NetworkDiffusive}
\dot x_k = \ff(x_k) + \sigma\sum_{j=1}^N A_{kj}\gf(x_j-x_k).
\end{align}
In general, diffusive coupling through a nonlinear term $\gf(x_j-x_k)$ is not separable. However, if the coupling is linear, that is, $\gf(x_j-x_k)= \sigma(x_j-x_k)$ for $\sigma\in\bbR$, or we consider the linearized dynamics of~\eqref{eq:NetworkDiffusive} around the synchronization manifold $\sset{x_1=\dotsb=x_d}$, we have a separable coupling function with $\gf^\iin(x) = x$ and $\gf^\res = 1$ since we can rewrite the global dynamics as
\begin{align*}
\dot x_k = \ff(x_k) - N\sigma x_k+ \eps \sigma \sum_{j=1}^N A_{kj} x_j.
\end{align*}

Note that if either~$\Ns(\gf^\res)$ or~$\Ns(\gf^\iin)$ contains an open set~$U$ in~$\cM$ this naturally induces a dead zone. More precisely, suppose that~$U$ is a maximal open set such that $U \subset \Ns(\gf^\res)$ and~$\gf^\iin$ does not vanish on any open set. 
Then $ U\times \cM \subset\cM^2$ is a dead zone for $\gf=\gf^\res \odot \gf^\iin$ (an \emph{input dead zone}). Conversely, if $U \subset \Ns(\gf^\iin)$ and~$\gf^\res$ does not vanish on any open set then $\cM \times U \subset\cM^2$ is a dead zone for $\gf=\gf^\res \odot \gf^\iin$ (an \emph{output dead zone}).

\section{Dead zones in weakly coupled oscillator networks}
\label{sec:Oscillators}

We now assume that the intrinsic dynamics of each node is oscillatory. Specifically, suppose that $\cM=\Er^d$ with $d > 1$ and the uncoupled node dynamics $\dot x = \ff(x)$, $x\in\Er^d$, gives rise to an asymptotically stable limit cycle solution~$\gamma(t)$ in~$\Er^d$ of minimal period~$\tau>0$ so that $\gamma(t+\tau) = \gamma(t)$ for all~$t>0$.
In other words, the uncoupled network contains a normally hyperbolic invariant torus~$\bbT^N$ that persists for weak coupling $|\eps|\ll 1$. The main idea of a \emph{phase reduction} is to approximate the dynamics of the coupled system~\eqref{eq:Network} by the evolution of phases $\theta = (\theta_1, \dotsc, \theta_N)\in\bbT^N$. This reduces the dimension of the phase space from~$\bbR^{dN}$ to~$\bbT^N$.  
Here we are interested how dead zones of the full system~\eqref{eq:Network} on~$\bbR^{dN}$ with oscillatory intrinsic dynamics given by~$\ff$ induce dead zones for the dynamics of the phase variables.

\subsection{Weak coupling and dead zones in phase oscillator networks}

Before we consider dead zones, we briefly review the main ingredients of a (first-order) phase reduction; for reviews of this well-established technique see References~\onlinecite{Izhikevich2007,AshCooNic2016,PD2019}.

\subsubsection{Phase response curves and phase reduction}

Consider a single uncoupled oscillator 
\begin{equation}\label{eq:SingleOsc}
\dot x = \ff(x)
\end{equation}
whose dynamics includes an asymptotically stable limit cycle~$\gamma(t)$ of minimal period~$\tau$; here we suppress the oscillator index. The set $\Gamma=\set{\gamma(t)}{t \in \Er}$ is a flow-invariant circle that we can parametrize using a phase variable~$\nud: \Gamma\to\bbT$ such that $\dot\nud = \omega$ with $\omega = 2\pi/\tau$; this function is invertible on the limit cycle. Indeed, for any point~$y_0$ with trajectory~$y(t)$ in the basin of attraction of~$\Gamma$, we define its asymptotic phase~$\nu(y_0):=\psi\in\bbT$ such that
\[\norm{\gamma(\psi/\omega+t)-y(t)}\to 0\]
as $t\to\infty$. More precisely, the \emph{isochron} for $\vartheta \in \bbT$ is the $(d-1)$-dimensional level set 
\[
\Theta_{\vartheta} = \set{x\in\Er^d}{\nud(x) = \vartheta}
\]
of the phase function. Isochrons are defined in the basin of attraction of the limit cycle. Sometimes, with abuse of notation, we write $\nu^{-1}(\psi)$ for the point on~$\Gamma$ with phase~$\psi$. 

\begin{defn}
The \emph{(infinitesimal) phase response curve} of the oscillator is the function
\[Z:\bbT\to\bbR^d,~~ \vartheta\mapsto\grad(\nud)|_{\nu^{-1}(\vartheta)}\]
where $\grad$ denotes the gradient.
\end{defn}

The phase response curve encodes how the phase of an oscillator changes with respect to an infinitesimal perturbation. Now suppose that the oscillator is subject to a weak forcing given by an input~$\If(t)$, i.e., consider
\[
\dot x = \ff(x) + \eps \If(t)+O(\eps^2)
\]
with~$\If$ bounded and $|\eps|\ll 1$. Expanding in the small parameter~$\eps$, the dynamics of the phase variable $\vartheta = \nu(x)$ close to~$\Gamma$ up to first order is
\begin{align}
    \dot\vartheta &= \omega+\eps \langle Z(\vartheta),\If(t) \rangle +O(\eps^2),
    \label{eq:PhaseDyn}
\end{align}
where $\langle\,\cdot\,,\cdot\,\rangle$ denotes the usual scalar product on~$\bbR^d$. In other words, the effect of the forcing on the phase---up to first order---is given by the projection of the forcing~$\gf$ on the phase response curve~$Z$.

Consider a network that consists of~$N$ coupled oscillatory units~\eqref{eq:Network}, that is, oscillator~$k$ is forced by the other oscillators according to
\[
\If_k(t) = \sum_{j=1,j\neq k}^N\gf(x_j, x_k).
\]
If all $x_k = \nud^{-1}(\theta_k)$ are close to~$\Gamma$ then we can define a function $\hat\gf:\bbT^2\to\bbR^d$ by $\hat\gf(\theta_j,\theta_k) := \gf(\nud^{-1}(\theta_j), \nud^{-1}(\theta_k))$. This yields the \emph{induced phase interaction function} $\gf^\PR:\bbT^2\to\bbR$ by
\begin{equation}\label{eq:gPR}
\begin{split}
\gf^\PR({\theta_j, \theta_k}) &:= \langle Z(\theta_k),\gf(\nud^{-1}(\theta_j), \nud^{-1}(\theta_k))\rangle\\&=\langle Z(\theta_k),\hat\gf(\theta_j, \theta_k)\rangle
\end{split}
\end{equation}
such that, with~\eqref{eq:PhaseDyn}, we can truncate at first order and the phase of oscillator~$k$ evolves approximately according to
\begin{equation}\label{eq:PhaseOsc}
\dot \theta_k  = \omega + \eps\sum_{j=1}^N \gf^\PR(\theta_j, \theta_k).
\end{equation}
Note that this system is a network dynamical system of the form~\eqref{eq:Network} in its own right on $\cM=\bbT$ with constant $\ff = \omega+\gf^\PR(0,0)$. Consequently, and in slight abuse of notation, we will write~$\gf^\PR$ just as~$\gf$ if it is clear from the context whether~$\gf$ is a function on~$\bbR^{2d}$ and, if not, whether it is induced by a phase reduction.

\subsubsection{Dead zones in weakly coupled phase oscillator networks}
\label{sec:PhaseRedDZ}

The phase reduction yields conditions for the phase reduced network~\eqref{eq:PhaseOsc} to have dead zones. The first result follows directly from the definition of~$\gf^\PR$:

\begin{lemm}\label{lem:DZphasered}
Suppose that~$\eps>0$ is sufficiently small and that the coupling function~$\gf(x_j,x_k)$ of the coupled oscillator network~\eqref{eq:Network} with limit cycle~$\Gamma$ has a dead zone~$U$ such that $\Gamma^2\cap U\neq \emptyset$. Then~$\gf^\PR(\theta_j,\theta_k)$ for the phase dynamics~\eqref{eq:PhaseOsc} has a dead zone for the set of phases in~$\bbT^2$ that has a nonempty intersection with~$\Gamma^2\cap U$.
\end{lemm}

\begin{proof}
Note that $\gf^\PR(\theta_j,\theta_k)=\langle Z(\theta_k),\hat\gf(\theta_j, \theta_k)\rangle$ and so if $(\nud^{-1}(\theta_j),\nud^{-1}(\theta_k))\in \Gamma^2\cap U$ then $\hat\gf(\theta_j,\theta_k)=\gf(\nud^{-1}(\theta_j),\nud^{-1}(\theta_k))=0$ and hence $\gf^\PR(\theta_j, \theta_k)=0$.
\end{proof}

The phase dynamics~\eqref{eq:PhaseDyn} also gives a further geometric condition for the emergence of a dead zone. Let $T_\vartheta := \Tg_{\nu^{-1}(\vartheta)}\Theta_{\vartheta}$ denote the tangent space of the isochron at phase~$\vartheta$. Since the phase response curve is the normal vector for the isochron, we have $\vp{v}{Z(\nud)} = 0$ for any $v\in T_\vartheta$. 

The next result for oscillator networks~\eqref{eq:Network} applies where the input acts in a fixed direction. This assumption is valid in many applications where the input acts on a particular component. It is straightforward to give a similar condition for arbitrary network coupling.

\begin{prop}\label{DZweak}
Consider the oscillator network~\eqref{eq:Network} for a coupling function $\gf(x_j,x_k) = \tilde\gf(x_j,x_k)v$ with a scalar function $\tilde\gf:\bbR^d\times\bbR^d\to\bbR$ and $v\in\bbR^d$ fixed.
If there is an interval $A\subset\bbT$ such that $v\in T_\vartheta$ for all $\vartheta\in A$ then the phase reduced system~\eqref{eq:PhaseOsc} with phase interaction function~$\gf^\PR(\theta_j,\theta_k)$ has a dead zone $U\in\bbT^2$ with $\bbT\times A\subset U$.
\end{prop}

\begin{proof}
For the assumed coupling we have
$\gf^\PR({\theta_j, \theta_k}) = \tilde\gf(\nud^{-1}(\theta_j), \nud^{-1}(\theta_k))\langle Z(\theta_k),v\rangle$ for the reduced system~\eqref{eq:PhaseOsc}.
Thus, $\langle Z(\theta_k),v\rangle = 0$ for $\theta_k\in A$ implies $\gf^\PR({\theta_j, \theta_k}) = 0$.
\end{proof}

\begin{rem}\label{rem:DZGeometry}
Note that this proposition gives sufficient conditions for a dead zone in the phase dynamics without having a dead zone in the original system, i.e., $\gf\neq 0$. In other words, Proposition~\ref{DZweak} says that if the network input is parallel to the isochrons at the limit cycle for an interval of phases, then this induces a dead zone. This is a \emph{geometry induced dead zone} for the phase dynamics. 
\end{rem}

\subsection{Dead zones for averaged identical phase oscillators}

If the oscillator forcing~$\If$ in~\eqref{eq:PhaseDyn} is periodic with approximately the same period as the oscillation itself, one can simplify the dynamics further through an averaging approximation; cf.~Reference~\onlinecite{Sanders2007} for general theory. This is particularly applicable in the case that the oscillator is weakly coupled to other identical oscillators through some network \cite{Ashwin1992}. The averaged system does not describe the oscillation in itself but slow variations of the oscillations relative to one another. Averaging leads to a diffusively coupled phase oscillator system~\eqref{eq:NetworkOscAvg} as we explain below. We give a brief overview of the averaging approximation before outlining sufficient conditions for dead zones to arise in the averaged system; the latter are the dead zones analyzed in Ref~\onlinecite{AshBicPoi2019}.

Averaging the system~\eqref{eq:PhaseOsc} over one period yields a phase oscillator network~\eqref{eq:NetworkOscAvg} with coupling through phase differences~\cite{Swift1992,AshCooNic2016}. More precisely, the averaged phase evolution, valid for small~$\eps$ and timescales~$t=O(1/\eps)$, is given by
\begin{equation}\label{eq:net-osci}
    \dot \theta_k = \omega + \eps\sum_{j=1}^N \hf(\theta_j-\theta_k)
\end{equation}
with coupling function
\begin{align}\label{eq:funH}
\begin{split}
    \hf(\vartheta)&=\dfrac{1}{2\pi}\int_0^{2\pi}\gf^\PR({s, \vartheta+s})\,\ud s\\
    &=\dfrac{1}{2\pi}\int_0^{2\pi}\left\langle Z(\vartheta+s),\hat\gf(s,\vartheta+s)\right\rangle\,\ud s
\end{split}
\end{align}    
Using linearity, we can also write
\begin{align}\label{eq:hComp}
\begin{split}
    \hf(\vartheta)&= \dfrac{1}{2\pi}\sum_{\ell=1}^d \int_0^{2\pi} \hat\hf_{\ell}(\vartheta,\vartheta+s)\,\ud s\\
     &\text{with~~} \hat\hf_{\ell}(\psi,\phi)= Z_{\ell}(\phi)\hat\gf_{\ell}(\psi,\phi),
     \end{split}
\end{align}
where the maps~$\hat\gf_{\ell}$ are the components of the interaction function~$\hat\gf(\theta_j,\theta_k)$.

In general we cannot expect a result analogous to Lemma~\ref{lem:DZphasered} to hold for the averaged dynamics~\eqref{eq:net-osci}: Even if there is an open interval on which either factor of the integrand in~\eqref{eq:funH} vanishes, the integral---and thus the resulting averaged coupling function---does not necessarily vanish. The following proposition gives a sufficient condition for there to be a dead zone for~$\hf$ and follows directly from consideration of~\eqref{eq:hComp}.

\begin{prop}\label{prop:res1}
Consider the oscillator network~\eqref{eq:net-osci} and suppose that $A\subset\bbT$ is an interval. If the set of phases in $\bbT^2$ where the components differ by elements in $A$ is contained in the zero set of all~$\hat\hf_{\ell}$, that is, if
\[
\set{(s,s-\vartheta)\in\bbT^2}{\vartheta\in A} \subset \bigcap_{\ell=1}^d \Ns\big(\hat\hf_{\ell}\big)
\]
then~$A\subset\DZ(\hf)$.
\end{prop}

\begin{rem}
In many models of relevance to applications, the network interactions act not only in a constant direction---as in Proposition~\ref{DZweak}---but also this direction is perpendicular to one of the coordinate axes. For example, for interacting neural oscillators the coupling is often through a single variable, namely the membrane voltage; we will discuss further explicit examples below. In this case, the condition in Proposition~\ref{prop:res1} simplifies as $\Ns(\hat\hf_{\ell})=\bbT^2$ for all but one~$\ell$.
\end{rem}

\begin{rem}
The decomposition of~$\hf$ in~\eqref{eq:hComp} implies that both~$Z_{\ell}$ and $\gf_{\ell}$ have a role in determining $\Ns(\hat\hf_{\ell})$ and thus~$\DZ(\hf)$. If we define 
\[\pi_j:\bbT^2\rightarrow \bbT\]
to be projection onto the $j$th component, $j\in\sset{1,2}$, note that $\Ns(\hat\hf_{\ell})=\Ns(Z_{\ell}\circ \pi_2)\cup \Ns(\hat \gf_{\ell})$. 
\end{rem}

\newcommand{\whZ}{\widehat{Z}}

\subsection{Overlapping arcs and dead zones for separable coupling functions}

Now assume that the interaction function~$\gf$ of the system~\eqref{eq:Network} is separable in the sense of Section~\ref{sec:sep}, i.e.,  $\hat\gf(\theta_j,\theta_k)=\hat\gf^\iin(\theta_j)\odot\hat\gf^\res(\theta_k)$. For ease of notation, assume first that the oscillator coupling only acts in one variable. Omitting the relevant index, the coupling function~\eqref{eq:funH} can be written as
\begin{align}\label{eq:funHsep}
\begin{split}
	\hf(\vartheta)&=\dfrac{1}{2\pi}\int_0^{2\pi} Z(\vartheta+s)\hat\gf^\res(\vartheta+s)\hat\gf^\iin(s)\,\ud s\\&=\dfrac{1}{2\pi}\int_0^{2\pi} \whZ(\vartheta+s)\hat\gf^\iin(s)\,\ud s
	\end{split}
\end{align}
where the scalar function~$\gf^\iin$ describes the input and 
\[\whZ(\vartheta) := Z(\vartheta)\gf^\res(\vartheta)\]
is the combined phase response. Our next result is a  geometric condition that guarantees a nontrivial dead zone of~$\hf$. We define $\rho_{\vartheta}:\bbT\rightarrow \bbT$ to be the translation $\rho_{\vartheta}(s)=(s+\vartheta)$. 

\begin{prop}
	Consider system~\eqref{eq:net-osci} with coupling function~\eqref{eq:funHsep}. Suppose that 
	$\Ns(\whZ)$ and $\Ns(\hat\gf^\iin)$ contain intervals of length $L_1>0$ and $L_2>0$ respectively. If $L_1+L_2>2\pi$ then there is a nontrivial dead zone for~$\hf$ with length at least $L_1+L_2-2\pi>0$.
	\label{prop:geom}
\end{prop}

\begin{proof}
If $(\alpha,\beta)\in\Ns(\whZ)$ then $(\alpha-\vartheta,\beta-\vartheta)\subset\Ns(\whZ\circ \rho_{\vartheta})$ by definition. This, together with the assumption on~$L_1$ and~$L_2$, implies that there is an interval $C\subset\bbT$ of length $L_1+L_2-2\pi>0$ such that $\Ns(\whZ\circ \rho_{\vartheta}) \cup \Ns(\hat\gf^\iin)=\bbT$ for any $\vartheta\in C$. Now for any $\vartheta\in C$ the integrand in~\eqref{eq:funHsep} vanishes. Thus, $\vartheta\in \DZ(\hf)$, which proves the assertion.
\end{proof}

To state a more general result of this nature we first introduce some notation for overlapping arcs in $\bbT=\bbR/2\pi\bbZ$. We write $\Pi:\bbR\rightarrow \bbT$ for $\Pi(x)=x~(\bmod ~2\pi)$, the covering map. Given any $\alpha\neq \beta$ in $\bbT$, we say $(\alpha,\beta)\subset \bbT$ is an {\em arc}\footnote{One can similarly define arcs $[\alpha,\beta]$, $[\alpha,\beta)$ and $(\alpha,\beta]$.} with \emph{first extremity}~$\alpha$ and \emph{last extremity}~$\beta$ if for any $\alpha'\in\Pi^{-1}(\alpha)$, we can choose $\beta'\in\Pi^{-1}(\beta)$ with $\alpha'<\beta'<\alpha'+2\pi$ and such that
$$
(\alpha,\beta)=\Pi\bigl((\alpha',\beta')\bigr).
$$
We say that arc~$C_1$ {\em overlaps with} arc~$C_2$ if there exists an arc $C=(\alpha,\beta)$ with $\alpha<\beta$ such that $C\subset C_1 \cap C_2$. We say~$C_1$ overlaps with~$C_2$ at \emph{first (respectively last) extremity} of~$C_1$ if the arc $C \subset C_1 \cap C_2$ contains the first (respectively last) extremity of~$C_1$.
Suppose $S\subset \bbT$. We say $(\alpha,\beta)\subset S$ is a \emph{maximal arc} of~$S$ if $(\alpha-\chi,\beta+\chi)\not\subset S$ for all $\chi>0$. As a generalization of Proposition~\ref{prop:geom} we have the following result. 

\begin{prop}
	\label{prop:DZweakly}
Consider the system~\eqref{eq:net-osci} and assume that the phase interaction function is separable as above so that
\begin{equation}
\hat\hf_{\ell}(\phi,\psi)=\whZ_{\ell}(\phi)\hat\gf^\iin_{\ell}(\psi)
\label{eq:seph}
\end{equation}
for all~$\ell\in\sset{1,\dotsc,d}$.
Suppose there exists an $\alpha>0$ such that for all~$\ell\in\sset{1,\dotsc,d}$ there are arcs $C_{\ell,1}\subset \Ns(\whZ_{\ell}\circ\rho_\alpha)$ and $C_{\ell,2}\subset \Ns(\hat\gf^\iin_{\ell})$ such that~$C_{\ell,1}$ and~$C_{\ell,2}$ overlap at last extremity of~$C_{\ell,1}$. Then the dead zone~$\DZ(\hf)$ for system~\eqref{eq:net-osci} contains an arc $(\alpha,\beta)$ for some $\beta\neq\alpha$.
\end{prop}

\begin{proof}
Using~\eqref{eq:seph} in~\eqref{eq:hComp} we can write $\hf(\vartheta)$ as
\begin{equation}
\hf(\vartheta)=\dfrac{1}{2\pi}\sum_{\ell=1}^d \int_0^{2\pi}(\whZ_{\ell}\circ\rho_{\vartheta})(s)\hat\gf^\iin_{\ell}(s) \,\ud s.
\label{eq:inth}
\end{equation}
By assumption, for $\vartheta=\alpha$ and all $1<\ell<d$ we have $\Ns(\whZ_{\ell}\circ \rho_\vartheta)\cup \Ns(\hat\gf^\iin_{\ell})=\bbT$. Moreover, as these intervals overlap on a non-empty interval there must be a non-empty arc $(\alpha,\beta)$ such that $\Ns(\whZ_{\ell}\circ\rho_\vartheta)\cup\Ns(\hat\gf^\iin_{\ell})=\bbT$ for all $\vartheta\in(\alpha,\beta)$ and $1\leq \ell\leq d$. Hence we have $(\alpha,\beta)\subset\DZ(\hf)$.
\end{proof}

There are other possible reasons why a dead zone may appear in~$\hf$ even if not present in~$\gf$ (or~$\gf^\PR$), in the case where the integral in~\eqref{eq:inth} cancels out for a range of~$\vartheta$. Effectively this can be thought of as the translates of~$Z_{\ell}$ being orthogonal functions to~$\gf_{\ell}$. We expect this is less likely to arise in applications in that it requires stipulations on global properties of these functions.

\section{Multiscale oscillators and dead zones}
\label{sec:multiscale}

In various applications---especially neuroscience~\cite{AshCooNic2016}---one is concerned with the behavior of coupled oscillators whose intrinsic dynamics~\eqref{eq:SingleOsc} have multiple timescales. In the following we consider the emergence of dead zones for such coupled multiscale oscillator networks; for the sake of clarity we consider two slow-fast oscillators, but note that there are obvious generalizations to $N \geq 2$ oscillators with more than two timescales. 

We consider the following specific example of~\eqref{eq:Network}:
\begin{equation}\label{eq:slow-fast-coup}
\begin{aligned}
\mu \dot v_k &= f(v_k,w_k)+\eps I(v_k,w_k,v_j,w_j),\\
\dot w_k &= g(v_k,w_k)
\end{aligned}
\end{equation}
for $k=1,2$, $j=3-k$
such that the state of oscillator~$k$ is given by $x_k = (v_k, w_k)\in\Er^2$ with fast variable~$v_k$ and slow variable~$w_k$. We assume the intrinsic dynamics are governed by sufficiently smooth $\ff = (f,g)$ and the ratio of intrinsic timescales $\mu \ll 1$, and $\gf = (I, 0)$ determines the interactions between the oscillators.

We recall some standard notions for such oscillators (see Reference~\onlinecite{Izhikevich2000} for more details); since we deal with a single oscillator, we omit the oscillator index~$k$. For the uncoupled oscillator $\eps=0$ the {\em singularly perturbed system}
\begin{equation}
\label{eq:slow-fast-0}
\begin{split}
	\mu\dot v &= f(v,w)\\
	\dot w &= g(v,w)
	\end{split}
\end{equation}
has slow-fast dynamics with {\em critical manifold} which is a one dimensional manifold
\[
\Ns(\ff)=\set{(v,w) \in \bbR^2}{f(v,w)=0}.
\]
The \emph{reduced system} is defined in the singular limit $\mu=0$ in terms of the differential algebraic equation
\begin{equation}
\label{eq:reduced}
\begin{aligned}
0 &= f(v,w)\\
\dot{w}&= g(v,w).
\end{aligned}
\end{equation}
In general, $\Ns(\ff)$ need not be a graph over~$w$: There may be a number of branches of solutions to $f(\xi(w),w)=0$, each parametrized continuously by~$w$. We assign them a stability according to the stability of~$\xi(w)$ as an equilibrium of the fast (layer) equation
\begin{equation}
\label{eq:layer}
\begin{aligned}
v' &= f(v,w)\\
w'&= 0,
\end{aligned}
\end{equation}
where $v':=\frac{\ud}{\ud s}v$ denotes the derivative with respect to slow time $s=\mu t$.

A singular solution of the reduced system~\eqref{eq:reduced} is a piecewise continuous solution that is continuous on any stable branch of solutions of~\eqref{eq:layer} \cite{Kuehn2015}. If the solution arrives at  a saddle-node (fold) at the end of such a branch, the layer equation defines a unique \emph{drop point} to a new stable branch.

We say a singular solution is a {\em simple relaxation oscillation with~$Q$ slow branches} (see, e.g., Reference~\onlinecite[Definition~5.2.4]{Kuehn2015}) if it is a periodic solution~$\gamma_0(t)$ with period~$\tau_0$ that consists of~$q$ alternating fast and slow segments, the jumps occur at generic fold points and drop points are normally hyperbolic. This implies we can partition $0=s_0<s_1<\cdots<s_Q=\tau_0$ and there are solutions~$u_q(t)$ of~\eqref{eq:reduced} such that $\gamma_0(t)=u_q(t)$ if $t\in(s_{q-1},s_q)$, $q\in\sset{2, \dotsc, Q}$. Standard results on relaxation oscillations (see Reference~\onlinecite{Izhikevich2000} or Reference~\onlinecite[Theorem 5.5.3]{Kuehn2015}) mean that for~$\mu$ close enough to zero there is a stable limit cycle~$\gamma_\mu(t)$ of~\eqref{eq:slow-fast-0} whose trajectory limits to~$\gamma_0(t)$ and such that the period~$\tau_{\mu}$ limits to~$\tau_0$ as $\mu\rightarrow 0$. Moreover, the durations $\tau_{\mu,q}$ spent close to the slow segment $u_q(t)$ tend to $\tau_{0,q}:=s_q-s_{q-1}$ as $\mu\rightarrow 0$.

Hence, in such a case there exists a stable limit cycle close to the simple relaxation oscillation for each oscillator in~\eqref{eq:slow-fast-0} in the uncoupled limit $\eps=0$. In the case of scale separation and weak coupling (i.e., where $|\eps|\ll \mu\ll 1$), Izhikevich \cite{Izhikevich2000} gives a reduction to phase equations of the form~\eqref{eq:PhaseOsc}, hence to the averaged system~\eqref{eq:net-osci}.

\subsection{Phase response of coupled slow-fast oscillators and dead zones of interaction}

\begin{figure}
		\includegraphics[width=1\linewidth]{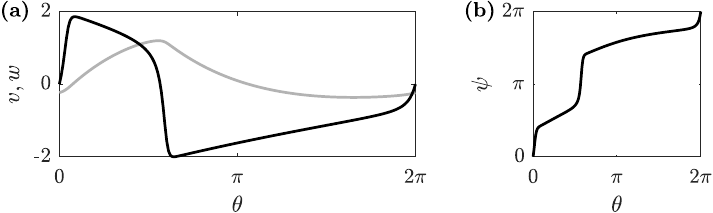}
	\caption{\label{fig:fhn_dz_f1}Panel~(a) show a single period of the limit cycle~$\gamma_{\mu}(t)$ for a FitzHugh--Nagumo oscillator~\eqref{eq:FHN1} with parameters~\eqref{eq:FHNparams}. The fast variable~$v$ is shown as a black line, the slow variable~$w$ as a gray line.
	The phase $\theta$ is chosen proportional to~$t$, such that after period~$\tau_\mu$, $\theta$ advances by~$2\pi$. We relate this to the geometric angle~$\psi$ modulo~$2\pi$ from~\eqref{eq:psitheta}. Panel~(b) shows the resulting function~$\psi(\theta)$. Observe the rapid changes in the fast variable as it switches between the two stable branches of the critical manifold.}	
\end{figure}

Proposition~\ref{prop:DZweakly} can be applied to show that System~\eqref{eq:net-osci} admits a dead zone in certain circumstances. We focus on an illustrative case of this for relaxation oscillation with two slow branches, where the coupling is localized to one of the slow branches. The example we consider is a pair of coupled Fitzhugh--Nagumo oscillators
\begin{equation}
\begin{aligned}\label{eq:FHN2}
\mu\dot v_1 &= v_1-\frac{1}{3}v_1^3-w_1 +i+\eps
I(v_1,w_1,v_2,w_2)\\
\dot w_1 &= v_1+a-b w_1 \\
\mu\dot v_2 &= v_2-\frac{1}{3}v_2^3-w_2 +i+\eps I(v_2,w_2,v_1,w_1)\\
\dot w_2 &= v_2+a-b w_2
\end{aligned}
\end{equation}
where we choose parameters
\begin{equation}
a=0.7, ~b=0.8, ~i=0.33, \mbox{ and }\mu=0.05,
\label{eq:FHNparams}
\end{equation}
and the coupling is mediated via some function~$I$ with coupling strength~$\eps$.  For $\eps=0$ the oscillators decouple into two systems of the form
\begin{equation}
\begin{aligned}\label{eq:FHN1}
\mu\dot v &= v-\frac{1}{3}v^3-w +i\\
\dot w &= v+a-b w.
\end{aligned}
\end{equation}
For the chosen parameters with $\mu=0$ the singular system has a simple relaxation oscillation which continues for small enough $\mu>0$ to give a stable limit cycle. We write this limit cycle as $(v,w)=:(V_\mu(t),W_\mu(t))=:\gamma_{\mu}(t)$ and the period as~$\tau_{\mu}$. Without loss of generality we assume $V_\mu(0)=0$ and $W_\mu(0)<0$. We define the phase on this limit cycle using $\theta(t)=t/\tau_{\mu}$ mod $1$, so that $\dot{\theta}=\omega=2\pi/\tau_{\mu}$ is constant and $(v,w)=(V_\mu(\theta \tau_{\mu}),W_\mu(\theta \tau_{\mu}))$. Figure~\ref{fig:fhn_dz_f1}(a) gives a numerical approximation of this limit cycle $\gamma_{\mu}$. All numerical computations are performed using the MATLAB {\tt ode45} integrator.

We can also define a geometric phase~$\psi$ mod $2\pi$ of the limit cycle $\gamma_{\mu}$ in the $(v,w)$-plane by recording the angle~$\gamma_\mu(t)$ makes to the line $v=v^\bullet$, $w<w^\bullet$ from a point $(v^\bullet,w^\bullet)$ within the limit cycle. This angle increases monotonically on the limit cycle. For small enough coupling the mapping between~$\theta$ and~$\psi$ is invertible and orientation preserving. More precisely, we compute
\begin{equation}
\psi = \tan^{-1}\left(\frac{V_\mu(\theta \tau_{\mu})-v^\bullet}{w^\bullet-W_\mu(\theta \tau_{\mu})}\right).
\label{eq:psitheta}
\end{equation}
The relationship between $\theta$ and~$\psi$ for~\eqref{eq:FHN1} is shown Figure~\ref{fig:fhn_dz_f1}(b) on choosing $(v^\bullet,w^\bullet)=(0,0.5)$. Observe the rapid change in $\psi$ during the fast transitions, while $\theta$ evolves at a constant speed.

\subsection{Sufficient conditions for dead zones in coupled slow-fast oscillators with separable coupling}

We start with a proposition that gives sufficient conditions for a dead zone in the phase reduced equations for coupled slow-fast oscillators of the form~\eqref{eq:slow-fast-coup} consisting of two branches where there is coupling only on one of the branches, and illustrate this for a specific example of coupled FitzHugh--Nagumo oscillators~\eqref{eq:FHN2}.

\begin{prop}
\label{prop:suff}
Suppose the uncoupled oscillators ($\eps=0$) of system~\eqref{eq:slow-fast-coup} have simple relaxation oscillations with two branches of period $\tau_0>0$ for given by $(V_0(t),W_0(t))$ and suppose the durations this limit cycle spends on first and second branch of the oscillation is $\alpha \tau_0$ and $(1-\alpha)\tau_0$. Suppose that the input is separable, i.e.,
\[
\If(v_1,w_1,v_2,w_2)=\gf^\res(v_1,w_1)\gf^\iin(v_2,w_2)
\]
and~$\gf^\res$ and~$\gf^\iin$ 
vanish on a neighborhood of the second branch; we refer to this as the \emph{dead branch}. Suppose that $0<\alpha<\frac{1}{2}$. Then there exists $\mu_0>0$ such that for any $0<\mu<\mu_0$ there is an $\eps_0>0$ (depending, in general, on $\mu$) such that for any $0<\eps<\eps_0$, the reduced coupled phase oscillator network~\eqref{eq:net-osci} has a dead zone.
\end{prop}

\begin{proof}
If $0<\alpha<\frac{1}{2}$ then the assumption of a simple relaxation oscillation in the singular limit means that for small enough~$\mu$  
there is a limit cycle $\gamma_\mu(t)$ with period~$\tau_\mu$ close to~$\gamma_0(t)$ such that the durations spent in a neighborhood of the second branch is greater than $\tau_{\mu}/2$. If~$\eps$ is small enough the phase reduction and averaging means we reduce to~\eqref{eq:net-osci} where
\[
\whZ(\psi_1)\hat{\gf}(\psi_2)=Z(\psi_1)\gf^\res(\gamma_\mu(\psi_1)) \gf^\iin(\gamma_\mu(\psi_2))
\]
is zero if either $\psi_1$ or $\psi_2$ are on the dead branch. Hence, for small enough $\mu$ and $\eps$, the proportion of time spent on the dead branch, Proposition~\ref{prop:geom} can be applied with~$L_1$ and~$L_2$ bigger than some $\chi$ for some $2\pi(1-\alpha)\geq\chi>\pi$. Hence $L_1+L_2\geq \chi >2\pi$ for small enough $\mu$ and $\eps$, meaning there is a dead zone of length at least $L_1+L_2-2\pi=2\chi-2\pi>0$ for~$\hf$.
\end{proof}

\begin{figure}
		\includegraphics[width=1\linewidth]{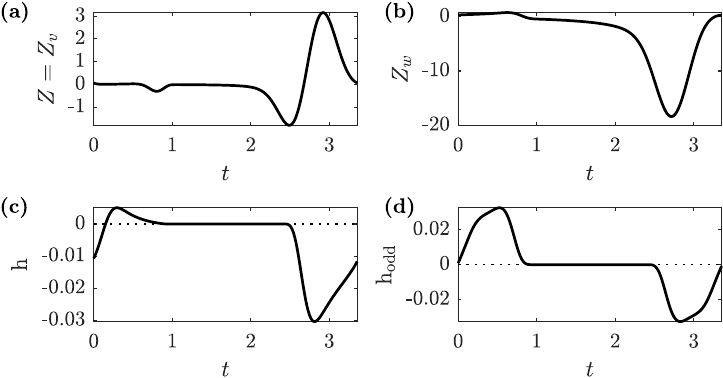}
	\caption{Top panels~(a) and~(b) show the bounded solution $Z(t)=(Z_v(t),Z_w(t))$ of the adjoint variational equation~\eqref{eq:FHNad2} for single FitzHugh--Nagumo oscillator~\eqref{eq:FHN1} with parameters~\eqref{eq:FHNparams} and $\eps=0$, for $t$ over one period of $\tau_{\mu}\approx 3.36$. Bottom panels~(c) and~(d) show the averaged phase interaction function~$\hf(\theta)$ and~$\hf_\textrm{odd}(\theta)$ for two weakly coupled oscillators~\eqref{eq:FHN2}.}
	\label{fig:fhn_dz_f2}
\end{figure}

To apply this we calculate the averaged phase equations using Malkin's method~\cite{Izhikevich2000}: this gives the infinitesimal phase response by computing the unique normalized bounded solution $Z(t)=(Z_v(t),Z_w(t))$ of the adjoint variational equation of~\eqref{eq:FHN1}, namely the periodic solution of
\begin{equation}
\begin{aligned}\label{eq:FHNad2}
\dot Z &= -\ud\ff^\mathsf{T}(\gamma_\mu(t)) Z.
\end{aligned}
\end{equation}
(where $\ud\ff^\mathsf{T}$ represents the transposed Jacobian for~\eqref{eq:FHN1}) that satisfies the condition
\[\left\langle Z(t),\dot\gamma_\mu(t)\right\rangle=1\]
for all~$t$; Izhikevich \cite{Izhikevich2000} gives expressions for this in the limit $\mu\rightarrow 0$. Since the coupling is in the first component only, we write~$Z=Z_v$ for simplicity.
A numerical approximation of the solution~$Z(t)$ of the adjoint variational equation is illustrated in Figure~\ref{fig:fhn_dz_f2}(a,b).

We consider a specific case of Proposition~\ref{prop:suff} where the system~\eqref{eq:FHN2} is coupled via
\begin{equation}
\label{eq:coupling}
\If(v_1,w_1,v_2,w_2)=\begin{cases}v_1v_2&\text{if $v_1$ and $v_2>0$,}\\ 0 &\text{otherwise}.\end{cases}
\end{equation}
This coupling acts in the first components only and is separable with identical input and response function. This choice of coupling clearly has a dead zone for the system~\eqref{eq:FHN2}; we demonstrate that this can lead to a dead zone for the phase reduced and averaged systems.

The averaged phase interaction for two coupled oscillators~\eqref{eq:FHN2} with coupling mediated by~\eqref{eq:coupling} can now be computed from~\eqref{eq:funH} as shown in Figure~\ref{fig:fhn_dz_f2}(c). This yields phase dynamics
\begin{equation}
\begin{aligned}
\dot{\theta}_1&=\omega +\eps \hf(\theta_2-\theta_1)\\
\dot{\theta}_2&=\omega +\eps \hf(\theta_1-\theta_2)\\
\end{aligned}
\end{equation}
where $\omega=2\pi/\tau_{\mu}$. Note the presence of constant zones in~$\hf(\theta)$: These become dead zones on choice of appropriate~$\omega$. Finally, Figure~\ref{fig:fhn_dz_f2}(d) shows the averaged phase difference for $\Phi=\theta_1-\theta_2$: Its evolution is governed by
\begin{equation}
\dot{\Phi}=-\eps \hf_\textrm{odd}(\Phi)
\end{equation}
where 
\[\hf_\textrm{odd}(\Phi)=\hf(\Phi)-\hf(-\Phi)\]
is twice the odd part of~$\hf$.

\begin{figure*}
		\includegraphics[width=0.49\linewidth]{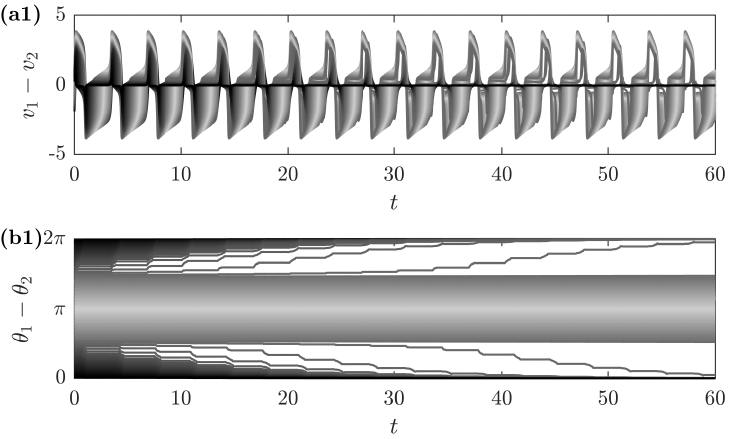}\hfill
		\includegraphics[width=0.49\linewidth]{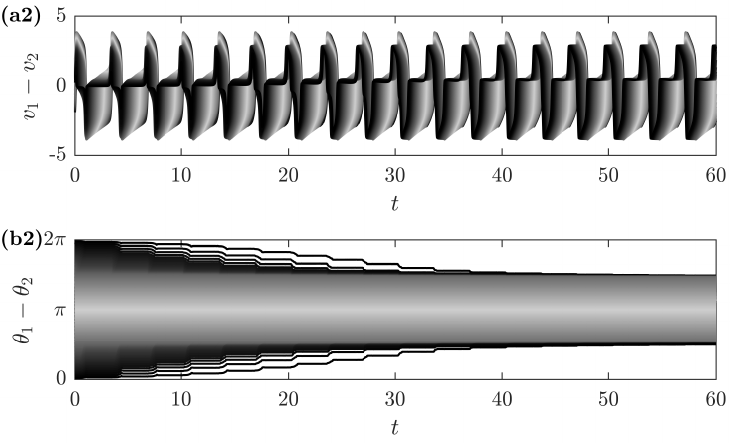}
	\caption{An ensemble of~$100$ evenly spaced phase differences $\theta_1-\theta_2$ on the uncoupled limit cycle are taken as initial conditions for two coupled FitzHugh--Nagumo systems~\eqref{eq:FHN2} with coupling~\eqref{eq:coupling} and parameters~\eqref{eq:FHNparams}
	 on a periodic colorscale.
	 The panels show the cases $\eps=0.04$ (case~1; left) and $\eps=-0.04$ (case~2; right). 
	 The top panels~(a) shows the difference $v_1-v_2$ for solutions starting in the ensemble of initial conditions. 
	 The bottom panels~(b) shows the evolution of the extracted phase differences $\theta_1-\theta_2=\Phi$ obtained by computing the geometric phase for each oscillator and inferring the phase $\theta_k$ for the corresponding uncoupled oscillator.
	 Both cases show the dead zone for the phase difference: Initial conditions with phase difference starting in an interval containing the antiphase solution $\Phi=\pi$ remain fixed while those outside this interval evolve towards the in-phase oscillation $\theta_1=\theta_2$ in case~1 and the boundary of the dead zone in case~2.}
	\label{fig:fhn_dz_f34}
\end{figure*}

For the chosen parameters, $\hf_\textrm{odd}$ has an dead zone for phases in a neighborhood of the antiphase solution.  Figure~\ref{fig:fhn_dz_f34} confirms the presence of this dead zone in simulations of the original equations~\eqref{eq:FHN2} with coupling~\eqref{eq:coupling} for coupling with $\eps=0.04$ and $-0.04$.
The top panels~(a) show $v_1-v_2$ for an ensemble of 100 initial conditions starting at evenly spaced phase differences. 
The bottom panels~(b) show the evolution of the difference $\Phi=\theta_1-\theta_2$ of the phases~$\theta_1$ and~$\theta_2$ from computing the geometric angles~$\psi_k(t)$ according to~\eqref{eq:psitheta}. Observe that for either sign of $\epsilon$ (in both cases~1 and~2) there is a band of initial conditions where the relative phases do not change. Note that the time series in panels~(a) do not show the dead zone as obvious: This is only clear after extraction of the phase angle. On examining the phase differences it becomes clear that case~1 has an attracting in-phase solution and is repelling on the boundary of the dead zone, while in case~2 the stabilities are reversed. The location of the dead zone agrees well with the averaged phase difference dynamics shown in the bottom right panel of Figure~\ref{fig:fhn_dz_f2}.

Keeping all parameters the same except for going to larger values of the parameter~$i$, the induced dead zone disappears (not shown) once the residence time on the dead branch becomes less than 1/2 of the cycle: This is because in this case the condition $L_1+L_2>2\pi$ in Proposition~\ref{prop:suff} no longer holds.

\subsection{Approximate Dead Zones for Additive Coupling}

In the previous section, separable coupling with identical input and response functions localized on one branch led to emergence of dead zones where the phase interaction function is exactly zero. For coupled neural oscillators, the interaction is often in a pulsatile way when the neuron ``fires''. We will now show that relaxation oscillators with pulsatile coupling give rise to \emph{approximate} dead zones in the averaged phase dynamics.

\begin{defn}\label{def:ApproxDZ}
Consider the network dynamical system~\eqref{eq:Network} and let $\eta>0$. An \emph{$\eta$-approximate dead zone} of the coupling function~$\gf$ is a maximal connected open set $U\subset\cM^2$ such that $\norm{\gf}_U\leq \eta$ where $\norm{\,\cdot\,}_U$ is the uniform norm on~$U$.
\end{defn}

Note that the notion of an $\eta$-approximate dead zone depend on the choice of norm on the tangent space.

We now consider a variation of the FitzHugh--Nagumo equations~\eqref{eq:FHN2} with pulsatile coupling. More specifically, the dynamics evolve according to
\begin{equation}
\begin{aligned}\label{eq:FHN3}
\dot v_1 &= v_1-\frac{1}{3}v_1^3-w_1 +i+\eps P(v_2)\\
\dot w_1 &= \mu (v_1+a-b w_1) \\
\dot v_2 &= v_2-\frac{1}{3}v_2^3-w_2 +i+\eps P(v_1)\\
\dot w_2 &= \mu (v_2+a-b w_2) 
\end{aligned}
\end{equation}
with parameters as above and~$P>0$ a pulse-like function, that is, its support is contained in some interval~$[\alpha, \beta]$ of phases on the limit cycle and $\frac{1}{2\pi}\int_\bbT P(\nu^{-1}(\vartheta))\,\ud\vartheta=1$. Note that the coupling is still separable, but the response function is constant and nonzero on the entire limit cycle. That means that the phase response is solely determined by the phase response curve~$Z$ of the individual unit.

The phase response curve~$Z$ in the fast variable has a specific form as shown in Reference~\onlinecite{Izhikevich2000}: In the singular limit of $\mu\to 0$ the~$Z$ converges pointwise to zero under the assumption that $|\eps| \ll \mu\ll 1$; this is illustrated in Figure~\ref{fig:fhn_dz_etaDZ}. 
That is, on the slow branch, the phase response is only nonzero close to the fast transitions along the orbit since by assumption the attraction to the slow branch is stronger than the coupling in the fast direction. Together with the pulsatile coupling, this now leads to the emergence of approximate dead zones.

This observation holds more generally and leads to the following result:

\begin{prop}\label{prop:etaDZ}
Suppose the uncoupled ($\eps=0$) oscillators of system~\eqref{eq:slow-fast-coup} have simple relaxation oscillations with~$Q$ slow branches.
Let $\eta > 0$ and suppose that the support $[\alpha,\beta]$ of the pulse function~$P$ is sufficiently narrow in the sense that its length is less than $\max_q 2\pi(\tau_{0,q}/\tau_0)$.
Then for sufficiently small~$\mu$ the coupling function~$\hf$ has an $\eta$-approximate dead zone for the averaged phase oscillator network~\eqref{eq:net-osci}.
\end{prop}

\begin{proof}
Define the arcs $C^\phi_q:=(t_{q}(2\pi/\tau_0)-\phi, t_{q+1}(2\pi/\tau_0)-\phi)\subset\bbT$ of phases on different segments of the critical manifold in the singular limit offset by~$\phi$.
By assumption, there is a $q'\in\sset{1, \dotsc, Q}$ and an open interval of phases $A\subset\bbT$ such that for $\psi\in A$ we have $[\alpha,\beta]\subset C^\psi_{q'}$ in terms of arcs on~$\bbT$. In other words, the support of the pulse function is sufficiently narrow to be fully contained in one of the segments of slow evolution in the singular limit.

We may now take~$\mu$ small enough such that $\abs{Z(\phi)}<\eta$ for $\phi\in C_{q'}$. With the definition of the averaged interaction function~\eqref{eq:funH} we have for any $\vartheta\in A$ that
\begin{align*}
    \abs{\hf(\vartheta)}&\leq\dfrac{1}{2\pi}\int_0^{2\pi}\abs{Z(\vartheta+s)}\abs{P(\nu^{-1}(s))}\,\ud s\\&
    \leq\dfrac{\eta}{2\pi}\int_{[\alpha,\beta]}\abs{P(\nu^{-1}(s))}\,\ud s \leq \eta,
\end{align*}
which proves the assertion that~$A$ is an $\eta$-approximate dead zone for~\eqref{eq:net-osci}.
\end{proof}

\begin{figure}
		\includegraphics[width=1\linewidth]{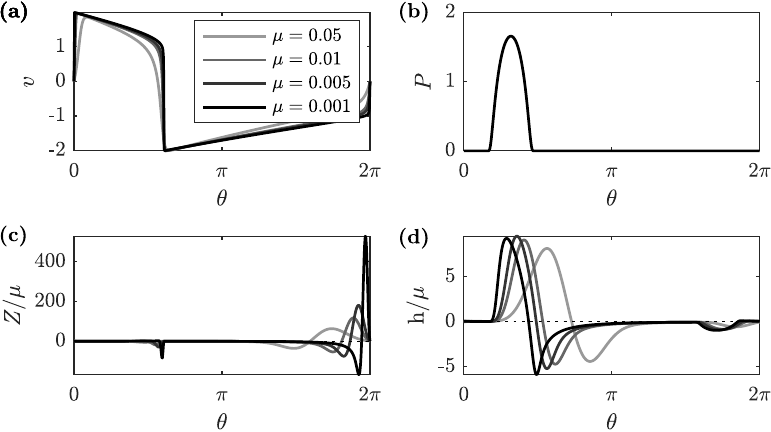}
	\caption{Approximate dead zones arise for coupled relaxation oscillators with pulsatile coupling~\eqref{eq:FHN3}. Panel~(a) shows the limit cycle oscillation for varying~$\mu$. Panel~(b) shows the pulse function in terms of the phase on the limit cycle. Panel~(c) shows the (first component of the) phase response curve as it converges to zero on the slow branches as $\mu\to 0$. The emergence of approximate dead zones can be seen in the resulting coupling function in Panel~(d) as~$\mu$ is decreased.}
	\label{fig:fhn_dz_etaDZ}
\end{figure}

We illustrate the emergence of $\eta$-approximate dead zones for coupled FitzHugh--Nagumo oscillators~\eqref{eq:FHN3} with pulsatile coupling. More explicitly, define the bump function
\[
\tilde P(x) := \begin{cases}\exp\left(-\frac{1}{1-x^2}\right)&\text{if} -1<x<1,\\
0& \text{otherwise}\end{cases}
\]
with support $[-1,1]$. With suitable normalization constant~$a$, scaling~$b=\frac{1}{2}$, and shift $c=1$ the function $P(\phi) = a \tilde P\big(\frac{\phi}{b}-c\big)$ with argument taken modulo~$2\pi$ is a pulse function with support $\big[\frac{1}{2}, \frac{3}{2}\big]$; see~Figure~\ref{fig:fhn_dz_etaDZ}(b). As the timescale parameter~$\mu$ is varied, the phase response function for the first component is converging pointwise to zero on the slow branches. This results in $\eta$-approximate dead zones according to Proposition~\ref{prop:etaDZ} for the pulse function~$P$, show in Figure~\ref{fig:fhn_dz_etaDZ}(d). Since the pulse is sufficiently narrow, there are actually two regions where the resulting coupling function~$\hf$ is small: First, for values $\theta\approx 0$ where the pulse aligns with the first slow branch, and around $\theta\approx 4$ when the shifted pulse aligns with the second branch.

\section{Discussion}
\label{sec:discuss}

We give some results that relate the presence of dead zones in the interaction of limit cycle oscillators~\eqref{eq:Network} with dead zones for reduced phase models~\eqref{eq:PhaseOsc}, valid in the weak coupling limit. In doing so we have highlighted that the connection may be subtle: There are cases where a dead zone for the former may or may not be inherited by the latter. Moreover, there are cases where a dead zone for the latter may not be associated with a dead zone of the former, although we suggest this is atypical. We give in Propositions~\ref{DZweak}, \ref{prop:res1}, \ref{prop:geom} and \ref{prop:DZweakly} some sufficient conditions for dead zones of averaged or non-averaged phase equations to result from dead zones of~\eqref{eq:Network}. We do not attempt to give necessary conditions and are not convinced this will be very transparent or instructive.
Although one could object to dead zones on the grounds of such constant sets are highly non-generic in smooth models, they can and do arise as a result of modelling assumptions. Moreover, approximate dead zones (regions where there is very little response) will be robust to small enough perturbations.

Given the periodicity of the phase coupling function $\hf(\vartheta)$  in $\vartheta$, an obvious approach \cite{Daido1996} is to consider a Fourier expansion of $\hf$. However, analyticity of any finite truncated Fourier expansion means that the only dead zones that will persist under truncation will be trivial. In general, the truncated Fourier representation of an $\hf$ with nontrivial dead zones will only have approximate dead zones.

We illustrate these mechanisms explicitly in Propositions~\ref{prop:suff} and \ref{prop:etaDZ} for an example of weakly coupled relaxation oscillators\cite{Izhikevich2000}, a dynamical system with two small parameters~\cite{Kuehn2020a}. The geometry that shapes the limit cycle oscillators yields an explicit calculation of phase response curves and thus allows for a concrete analysis of emergent dead zones. One ingredient for the emergence of dead zones is that there are regions where the phase response is trivial; this is not the case if the phase response is sinusoidal, for example, close to a Hopf bifurcation. On the other hand, dead zones could arise naturally in coupled piecewise continuous models of coupled oscillators \cite{AshCooNic2016} (such as the McKean model) where the coupling is also defined piecewise. It would be interesting to elucidate the emergence of dead zones for other weakly coupled but strongly nonlinear oscillations, such as limit cycles close to or emerging from homoclinic or heteroclinic structures~\cite{Homburg2010}.

Turning to forced rather than coupled oscillators, phase response of impulsively forced oscillator is another context where dead zones may be useful to understand circumstances where the forcing may or may not have an effect. For example, the model of temporally forced circadian transcriptional oscillators is shown in Reference~\onlinecite{UriTei2019} to have a region along the oscillation where the phase response (almost) vanishes; we argue that this would lead to dead zones if such oscillators are coupled, for example, as shown in Proposition~\ref{prop:geom}. Their model of a \emph{Drosophia} circadian clock is an interaction of three species, such that a coefficient multiplying the input is effectively zero for part of the oscillation. This lack of phase sensitivity at certain phases may be of biological utility if it allows interaction with the environment only for part of the cycle.

In this paper we consider only pairwise interaction of systems. It will be interesting to understand the role of dead zones in coupled dynamical systems with multi-way interactions (see Remark~\ref{rem:DZGeometry}).
Similarly, dead zones for approximations of the phase dynamics beyond first order (e.g., Reference~\onlinecite{Leon2019a}) will have higher order features of the geometry of the isochrons (curvature, etc) that will play a role.

Finally, it may be interesting to examine the effects of dead zones on coupled chaotic oscillators where no phase reduction is possible but nonetheless synchronization can occur~\cite{anishchenko1992synchronization}. Similarly, forced coupled oscillator systems~\cite{Anishchenko2009} have the potential for dead zones in the coupling and/or forcing.

\subsection*{Acknowledgements}

The research of PA and CP was funded by EPSRC Centre for Predictive Modelling in Healthcare, grant number EP/N014391/1. CB received funding by the EPSRC through grant EP/T013613/1.

\subsection*{Data Availability Statement}

The MATLAB code that generates the data in support the findings of this study is openly available in  {\tt https://github.com} in the repository {\tt /peterashwin/dead-zone-reduction-2021}.



\bibliographystyle{unsrt}
\def\urlprefix{}
\def\url#1{}

\bibliography{DZNeur_refs}

\end{document}